\documentclass[submission,copyright,creativecommons]{eptcs}
\providecommand{\event}{ACT 2020} 
\usepackage{breakurl}             % Not needed if you use pdflatex only.
\usepackage{underscore}           % Only needed if you use pdflatex.

\title{A diagrammatic approach to symmetric lenses}
\author{Bryce Clarke\thanks{The author is supported by the 
Australian Government Research Training Program Scholarship.}
\institute{Centre of Australian Category Theory\\
Macquarie University, Australia}
\email{bryce.clarke1@hdr.mq.edu.au}
}

\def\titlerunning{A diagrammatic approach to symmetric lenses}
\def\authorrunning{Bryce Clarke}

\usepackage[shortlabels]{enumitem}

\usepackage{mathtools, amssymb, amsthm}
\usepackage{eucal}

\usepackage{tikz-cd}

\newtheorem{theorem}{Theorem}[section]
\newtheorem{lemma}[theorem]{Lemma}
\newtheorem{proposition}[theorem]{Proposition}
\newtheorem{corollary}[theorem]{Corollary}

\theoremstyle{definition}
\newtheorem{defn}[theorem]{Definition}
\newtheorem{example}[theorem]{Example}

\theoremstyle{remark}

\newtheorem{notation}[theorem]{Notation}

\numberwithin{equation}{section}

\newcommand{\Cat}{\mathrm{\mathcal{C}at}}
\newcommand{\Cof}{\mathrm{\mathcal{C}of}}
\newcommand{\Meal}{\mathrm{\mathcal{M}eal}}
\newcommand{\Span}{\mathrm{\mathcal{S}pan}}
\newcommand{\SpnLens}{\mathrm{\mathcal{S}pn\mathcal{L}ens}}
\newcommand{\SymLens}{\mathrm{\mathcal{S}ym\mathcal{L}ens}}
\newcommand{\Lens}{\mathrm{\mathcal{L}ens}}
\newcommand{\SOpf}{\mathrm{\mathcal{SO}pf}}
\newcommand{\Mnd}{\mathsf{Mnd}}

\newcommand{\D}{\mathcal{D}}
\newcommand{\M}{\mathcal{M}}
\newcommand{\E}{\mathcal{E}}

\newcommand{\op}{\mathrm{op}}

\newcommand{\phibar}{\overline{\varphi}}
\newcommand{\br}{\rightleftharpoons}

\DeclareMathOperator{\ob}{ob}
\DeclareMathOperator{\cod}{cod}

\begin{document}
\maketitle

%----------------------------------------------------------------------%
% Abstract                                                             %
%----------------------------------------------------------------------%
\begin{abstract}
Lenses are a mathematical structure for maintaining consistency between
a pair of systems.
In their ongoing research program, Johnson and Rosebrugh have sought to
unify the treatment of symmetric lenses with spans of asymmetric lenses.
This paper presents a diagrammatic approach to symmetric lenses between 
categories, through representing the propagation operations with Mealy 
morphisms. 
The central result of this paper is to demonstrate that the bicategory 
of symmetric lenses is locally adjoint to the bicategory of spans of 
asymmetric lenses, through constructing an explicit adjoint triple 
between the hom-categories. 
\end{abstract}

%----------------------------------------------------------------------%
% Section 1: Introduction                                              %
%----------------------------------------------------------------------%
\section{Introduction}
\label{sec:introduction}

Lenses are a mathematical structure which model sychronisation between
a pair of systems. 
Lenses have been actively studied in both computer science and 
category theory since the seminal paper \cite{FGMPS07},
and now play an important role in a diverse range of applications 
including bidirectional transformations, model-driven engineering, 
database view-updating, systems interoperations, data sharing, and 
functional programming. 

While typically lenses are used to describe \emph{asymmetric}
relationships between systems, many of these examples are better 
understood as special cases of \emph{symmetric lenses}. 
Since the introduction of symmetric lenses in the paper \cite{HPW11},
there has been a significant research program lead by Johnson and 
Rosebrugh (see \cite{JR14, JR15, JR16, JR17art, JR17pro}) to unify their 
treatment with spans of asymmetric lenses. 
However, despite this research revealing numerous important aspects of 
symmetric lenses, many constructions appear ad~hoc by relying upon 
justification from applications, and remain without a robust 
category-theoretic foundation.

This paper develops a diagrammatic approach to symmetric lenses in 
category theory, which clarifies and generalises the previous results by
Johnson and Rosebrugh.
Symmetric lenses are characterised as a pair of \emph{Mealy morphisms}, 
and may be represented as certain diagrams in $\Cat$. 
The main result demonstrates, for a pair of 
systems $A$ and $B$, an adjoint triple between the category of symmetric
lenses and the category of spans of asymmetric lenses.
\begin{equation}
\label{eqn:main}
\begin{tikzcd}[column sep = large]
\SymLens(A, B)
\arrow[r, shift right = 5, ""{name = D}, hook]
\arrow[from = r, ""{name = M}]
\arrow[r, shift left = 5, ""{name = U}, hook]
& \SpnLens(A, B)
\arrow[phantom, from = M, to = D, "\bot"]
\arrow[phantom, from = M, to = U, "\bot"]
\end{tikzcd}
\end{equation}
Furthermore, these adjunctions characterise $\SymLens(A, B)$ as both 
a reflective and coreflective subcategory of $\SpnLens(A, B)$, 
and underlie local adjunctions between the corresponding bicategories 
$\SymLens$ and 
$\SpnLens$. 

This paper treats a system as a category, whose objects are the 
\emph{states} of the system, and whose morphisms are the \emph{updates}
(or transitions) between states of the system. 
In the paper \cite{DXC11}, \emph{asymmetric delta lenses} were introduced
as the maps between systems, which propagate updates in one system to
updates in another. 
A close category-theoretic study of delta lenses appeared in 
\cite{JR13}, and in \cite{AU17} it was discovered that they may be 
understood in terms of functors and cofunctors. 
In the ACT2019 paper \cite{Cla19}, delta lenses were generalised to 
internal category theory, and more importantly, it was shown that an 
asymmetric delta lens may be represented as a certain commutative
diagram in $\Cat$. 

The focus of this paper is \emph{symmetric delta lenses}, introduced in 
\cite{DXCEHO11}, and their relationship with spans of asymmetric delta
lenses.
While the key results are concentrated on the theoretical foundation of 
these structures, this work also contains important benefits towards
applications, from simplifying the use, and further study, of lenses.  

%----------------------------------------------------------------------%
% Subsection: Overview of the paper                                    %
%----------------------------------------------------------------------%
\subsection*{Overview of the paper}

Section~\ref{sec:background} reviews the different kinds of morphisms 
between categories which are later used to define lenses. 
The definitions of discrete opfibration, bijective-on-objects functor,
and fully faithful functor are recalled, as well as the less familiar
definitions of cofunctor (see \cite{HM93, Agu97}) and Mealy morphism
(see \cite{Par12}; also known as a \emph{two-dimensional partial map}
in \cite{LS02, Str20}). 
In Lemma~\ref{lem:cofunrep} and Lemma~\ref{lem:Mealy}, cofunctors and 
Mealy morphisms are faithfully represented as spans in $\Cat$.
Note that Mealy morphisms in this paper are slightly different from
\cite[Example~3]{Par12}, as there is no requirement for the functor 
component to be ``objectwise constant on the fibres''. 
The bicategory $\Meal$ of small categories and Mealy morphisms is 
equivalent to the bicategory $\Mnd(\Span)$ of monads and lax monad 
morphisms in the bicategory of spans. 

In Section~\ref{sec:asymmetriclenses}, the definition of an asymmetric 
lens is recalled from \cite{DXC11, JR13}, and their characterisation 
from \cite{Cla19} as diagrams in $\Cat$ is stated in 
Lemma~\ref{lem:lensrep}.
While the category $\Lens$ of small categories and lenses does not admit 
all pullbacks, it is proved in Proposition~\ref{prop:lensprod} that the 
category $\Lens(B)$, of lenses over a base category $B$, has products. 
Using this proposition, the bicategory $\SpnLens$ of small categories 
and spans of asymmetric lenses is constructed. 
From the perspective of applications, the bicategory $\SpnLens$ allows
the modelling of updates between systems which cannot synchronise 
directly, but instead depend on some intermediary system. 

Section~\ref{sec:symmetriclenses} presents a concise construction of 
the bicategory $\SymLens$ of small categories and symmetric lenses,
using the bicategory $\Meal$. 
A symmetric lens between a pair of systems may be understood as a set
of \emph{correspondences} between the states of the systems, together
with a pair of Mealy morphisms which propagate the system 
updates in each direction. 
Informally, the ``symmetric'' aspect of symmetric lenses may be 
understood in the context of dagger categories, through a canonical
family of functors 
$\dagger \colon \SymLens(A, B) \rightarrow \SymLens(B, A)$ which 
take the opposite of a symmetric lens. 

In Section~\ref{sec:adjointtriple}, the precise categorical relationship
between $\SpnLens$ and $\SymLens$ is presented by the adjoint 
triple in Theorem~\ref{thm:main}. 
The proof relies on the diagrammatic approach to symmetric lenses in 
an essential way, and reveals several aspects of \cite[Theorem~40]
{JR17art} which were hidden by an unnecessary equivalence relation.   

%----------------------------------------------------------------------%
% Section 2: Background									               %
%----------------------------------------------------------------------%
\section{Background}
\label{sec:background}

Let $\Cat$ denote the category of small categories and functors. 
There are three classes of functors that will be of particular 
interest in this paper. 

\begin{defn}
A functor $f \colon A \rightarrow B$ is a \emph{discrete opfibration}
if for all objects $a \in A$ and morphisms 
$u \colon fa \rightarrow b \in B$, there exists a unique morphism 
$\varphi(a, u) \colon a \rightarrow p(a, u)$ in $A$ such that 
$f \varphi(a, u) = u$. 
The notation $p(a,u)$ is used to denote the object 
$\cod(\varphi(a, u))$. 
Let $\D$ denote the class of discrete opfibrations. 
\end{defn}

\begin{defn}
A functor is \emph{bijective-on-objects} if its object assignment is a 
bijection. Let $\E$ denote the class of bijective-on-objects functors.
\end{defn}

\begin{defn}
A functor $f \colon A \rightarrow B$ is \emph{fully faithful} 
if for all objects $a, a' \in A$ and morphisms 
$u \colon fa \rightarrow fa' \in B$, 
there exists a unique morphism $w \colon a \rightarrow a'$ in $A$ such 
that $fw = u$. 
Let $\M$ denote the class of fully faithful functors. 
\end{defn}

There is a well-known orthogonal factorisation system $(\E, \M)$ on $\Cat$, 
called the \emph{bo-ff factorisation system}, in which every functor 
factorises into a bijective-on-objects functor followed by a fully 
faithful functor. 
The \emph{image} of a functor $f \colon A \rightarrow B$ is
a category $I_{f}$ whose objects are those of $A$, and whose morphisms
are triples $(a, u , a') \colon a \rightarrow a'$ where 
$a, a' \in A$ and $u \colon fa \rightarrow fa' \in B$. 
The functor $f \colon A \rightarrow B$ factorises through the image
as follows: 
\begin{equation}
\label{eqn:boff}
\begin{tikzcd}[column sep = large]
A
\arrow[r]
& I_{f}
\arrow[r]
& B \\[-1.5em]
a
\arrow[d, "w"']
& a 
\arrow[l, phantom, "\cdots"]
\arrow[r, phantom, "\cdots"]
\arrow[d, "{(a, fw, a')}"]
& fa 
\arrow[d, "fw"] \\
a' & a' 
\arrow[l, phantom, "\cdots"]
\arrow[r, phantom, "\cdots"]
& fa'
\end{tikzcd}
\end{equation}

The universal property of the bo-ff factorisation system may be stated
as follows. 
Given a commutative square of functors, 
\begin{equation}
\label{eqn:boffuniprop}
\begin{tikzcd}
A 
\arrow[d, "e"']
\arrow[r, "f"]
& B 
\arrow[d, "m"]
\\
C 
\arrow[r, "g"']
\arrow[ru, dashed, "h"]
& D
\end{tikzcd}
\end{equation}
where $e$ is bijective-on-objects and $m$ is fully faithful, there 
exists a unique functor $h \colon C \rightarrow B$ such that 
$h \circ e = f$ and $m \circ h = g$. 
In particular, note that this universal property defines the 
image $I_{f}$ uniquely up to isomorphism. 

\begin{defn}[See \cite{Agu97}]
\label{defn:cofunctor}
Let $A$ and $B$ be categories. 
A \emph{cofunctor} $\varphi \colon B \nrightarrow A$ consists of an 
assignment on objects $\varphi \colon \ob(A) \rightarrow \ob(B)$ together 
with an operation assigning each pair $(a, u)$, where $a \in A$ and 
$u \colon \varphi a \rightarrow b \in B$, to a morphism 
$\varphi(a, u) \colon a \rightarrow p(a, u)$ in $A$, satisfying the 
axioms: 
\begin{enumerate}[(1), itemsep = +1ex]
\item $\varphi p(a, u) = b$ 
\item $\varphi(a, 1_{\varphi a}) = 1_{a}$ 
\item $\varphi(a, v \circ u) = \varphi(p(a, u), v) \circ \varphi(a, u)$ 
\end{enumerate}
The notation $p(a,u)$ is used to denote the object 
$\cod(\varphi(a, u))$. 
\end{defn}

\begin{example}
Every discrete opfibration $A \rightarrow B$ yields a cofunctor 
$B \nrightarrow A$, and every bijective-on-objects functor 
$A \rightarrow B$ yields a cofunctor $A \nrightarrow B$. 
\end{example}

Let $\Cof$ denote the category of small categories and cofunctors. 
Given cofunctors $\gamma \colon C \nrightarrow B$ and 
$\varphi \colon B \nrightarrow A$, their composite 
$\varphi \circ \gamma \colon C \nrightarrow A$ may be understood from 
the following diagram: 
\begin{equation}
\label{eqn:cofuncomp}
\begin{tikzcd}[column sep = large]
C
\arrow[r, "/"{anchor=center,sloped}]
& B
\arrow[r, "/"{anchor=center,sloped}]
& A \\[-1.5em]
\gamma \varphi a
\arrow[d, "u"']
& \varphi a 
\arrow[l, phantom, "\cdots"]
\arrow[r, phantom, "\cdots"]
\arrow[d, "{\gamma(\varphi a, u)}"]
& a
\arrow[d, "{\varphi(a, \gamma(\varphi a,u))}"] \\
c 
& q(\varphi a, u) 
\arrow[l, phantom, "\cdots"]
\arrow[r, phantom, "\cdots"]
& p(a, \gamma(\varphi a,u))
\end{tikzcd}
\end{equation}

There is an orthogonal factorisation system $(\D^{\op}, \E)$ on $\Cof$, 
in which every cofunctor factorises into a discrete opfibration 
(taken in the opposite direction) followed 
by a bijective-on-objects functor.
The \emph{image} of a cofunctor $\varphi \colon B \nrightarrow A$ 
is a category $\Lambda$ whose objects are those of $A$, 
and whose morphisms are pairs $(a, u) \colon a \rightarrow p(a, u)$,
where $a \in A$ and $u \colon \varphi a \rightarrow b \in B$. 
The cofunctor $\varphi \colon B \nrightarrow A$ factorises through 
the image as follows: 
\begin{equation}
\label{eqn:cofunfac}
\begin{tikzcd}[column sep = large]
B
\arrow[r, "/"{anchor=center,sloped}]
& \Lambda
\arrow[r, "/"{anchor=center,sloped}]
& A \\[-1.5em]
\varphi a
\arrow[d, "u"']
& a 
\arrow[l, phantom, "\cdots"]
\arrow[r, phantom, "\cdots"]
\arrow[d, "{(a, u)}"]
& a
\arrow[d, "{\varphi(a, u)}"] \\
b
& p(a, u) 
\arrow[l, phantom, "\cdots"]
\arrow[r, phantom, "\cdots"]
& p(a, u)
\end{tikzcd}
\end{equation}
Notice that the cofunctor $B \nrightarrow \Lambda$ describes a discrete
opfibration $\Lambda \rightarrow B$, and the cofunctor 
$\Lambda \nrightarrow A$ describes an identity-on-objects functor 
$\Lambda \rightarrow A$. 

\begin{lemma}
\label{lem:cofunrep}
Given a cofunctor $\varphi \colon B \nrightarrow A$ there is a span of 
functors, 
\begin{equation}
\label{eqn:cofunrep}
\begin{tikzcd}[row sep = small, column sep = small]
& \Lambda 
\arrow[ld, "\phibar"']
\arrow[rd, "\varphi"]
& \\
B & & A
\end{tikzcd}
\end{equation}
where $\phibar$ is a discrete opfibration and $\varphi$ is 
identity-on-objects.
\end{lemma}
If $\Cof$ is understood as a locally-discrete $2$-category, 
Lemma~\ref{lem:cofunrep} provides a way of constructing a locally 
fully faithful, identity-on-objects pseudofunctor 
$\Cof \rightarrow \Span(\Cat)$. 
From now on cofunctors will always be given by their span representation
\eqref{eqn:cofunrep}. 

\begin{defn}[See \cite{Par12}]
\label{defn:Mealymorphism}
Let $A$ and $B$ be categories. 
A \emph{Mealy morphism} $A \nrightarrow B$ consists of a discrete 
category $X_{0}$ together with a span of functors
$(g_{0}, X_{0}, f_{0}) \colon A \nrightarrow B$ 
and operations assigning each pair $(x, u)$, where $x \in X_{0}$ and 
$u \colon g_{0}x \rightarrow a \in A$, to an object $q(x, u)$ in $X_{0}$
and a morphism $f(x, u) \colon f_{0}x \rightarrow f_{0}q(x, u)$ in $B$,
satisfying the axioms: 
\begin{enumerate}[(1), itemsep = +1ex]
\item $g_{0}q(x, u) = a$ 
\item $q(x, 1_{gx}) = x$ and $f(x, 1_{g_{0}x}) = 1_{f_{0}x}$
\item $q(x, v \circ u) = q(q(x, u), v)$ 
		and $f(x, v \circ u) = f(q(x, u), v) \circ f(x, u)$
\end{enumerate}
\end{defn}

\begin{example}
Every functor $A \rightarrow B$ yields a Mealy morphism 
$A \nrightarrow B$, and every cofunctor $B \rightarrow A$ yields a 
Mealy morphism $B \nrightarrow A$.  
\end{example}

\begin{example}[Example~4 in \cite{Par12}]
Given a pair of sets $A$ and $B$, a Mealy morphism between free monoids 
$A^{\ast}$ and $B^{\ast}$ is exactly a \emph{Mealy machine} with input 
alphabet $A$ and output alphabet $B$.
\end{example}

Let $\Meal$ denote the \emph{bicategory} of small categories and Mealy 
morphisms. Unlike the special cases functors and cofunctors, composition
of Mealy morphisms is not strictly associative, since the structure 
involves a span of functions. 
There are two possible notions of $2$-cell between Mealy morphisms; 
this paper uses the stricter notion as given below. 

\begin{defn}
\label{defn:mealy}
Let $(X_{0}, g_{0}, f_{0}, q, f)$ and $(Y_{0}, k_{0}, h_{0}, p, h)$ be 
Mealy morphisms $A \nrightarrow B$. 
A \emph{map of Mealy morphisms} consists of a morphism of spans,
\begin{equation}
\begin{tikzcd}[row sep = small, column sep = small]
& X_{0}
\arrow[ld, "g_{0}"']
\arrow[rd, "f_{0}"]
\arrow[dd, "m"]
& \\
A
& & B
\\
& Y_{0}
\arrow[lu, "k_{0}"]
\arrow[ru, "h_{0}"']
& 
\end{tikzcd}
\end{equation}
such that $mq(x, u) = p(mx, u)$ and $f(x, u) = h(mx, u)$ for each pair 
$(x, u)$, where $x \in X_{0}$ and $u \colon g_{0}x \rightarrow a \in A$.
\end{defn}

Analogous to the orthogonal factorisation system on $\Cof$, every 
Mealy morphism factorises into a discrete opfibration followed by 
a functor. 
Using the notation of Definition~\ref{defn:mealy}, the \emph{image} of a
Mealy morphism $A \nrightarrow B$ is a category $X$, whose 
set of objects is $X_{0}$ and whose morphisms are pairs 
$(x, u) \colon x \rightarrow q(x, u)$, where $x \in X_{0}$ and 
$u \colon g_{0}x \rightarrow a$. 
The factorisation of a Mealy morphism may then be described by the 
following diagram: 
\begin{equation}
\begin{tikzcd}[column sep = large]
A
\arrow[r, "/"{anchor=center,sloped}]
& X
\arrow[r, "/"{anchor=center,sloped}]
& B \\[-1.5em]
g_{0} x
\arrow[d, "u"']
& x 
\arrow[l, phantom, "\cdots"]
\arrow[r, phantom, "\cdots"]
\arrow[d, "{(x, u)}"]
& f_{0}x 
\arrow[d, "{f(x, u)}"] \\
a & q(x, u) 
\arrow[l, phantom, "\cdots"]
\arrow[r, phantom, "\cdots"]
& f_{0}q(x, u) 
\end{tikzcd}
\end{equation}
Notice that the Mealy morphism $A \nrightarrow X$ describes a discrete
opfibration $X \rightarrow A$, and the Mealy morphism $X \nrightarrow B$
describes a functor $X \rightarrow B$.

\begin{lemma}
\label{lem:Mealy}
Given a Mealy morphism $A \nrightarrow B$ there is a span of functors,
\begin{equation}
\label{eqn:Mealyrep}
\begin{tikzcd}[row sep = small, column sep = small]
& X 
\arrow[ld, "\overline{g}"']
\arrow[rd, "f"]
& \\
A & & B
\end{tikzcd}
\end{equation}
where $\overline{g}$ is a discrete opfibration.
\end{lemma}
Lemma~\ref{lem:Mealy} provides a way of constructing a 
locally fully faithful, identity-on-objects pseudofunctor 
$\Meal \rightarrow \Span(\Cat)$. 
From now on Mealy morphisms will always be understood by their span 
representation \eqref{eqn:Mealyrep}. 
It is also worth noting that every Mealy morphism may also be factorised
into a cofunctor followed by a fully faithful functor. 
These two possible factorisations would amount to a kind of ternary 
factorisation system $(\D^{\op}, \E, \M)$ on $\Meal$, however this 
observation won't be pursued in this paper. 

%----------------------------------------------------------------------%
% Section 3: Spans of asymmetric lenses                                %
%----------------------------------------------------------------------%
\section{Spans of asymmetric lenses}
\label{sec:asymmetriclenses}

The goal of this section is to introduce the following three structures: 
\begin{itemize}[$\diamond$, itemsep = +1ex]
\item The category $\Lens$ of small categories and (asymmetric) lenses; 
\item The category $\Lens(B)$ of lenses over a base category $B$;
\item The bicategory $\SpnLens$ of small categories and spans of lenses.
\end{itemize}
It is well-known that $\Lens$ does not have all pullbacks, which 
complicates the usual construction of the bicategory of spans. 
The main obstruction is that while every cospan in $\Lens$ admits a 
canonical cone, the universal property of the pullback does not hold. 
However, for any small category $B$, there is a suitably defined 
category $\Lens(B)$ which does admit cartesian products. 
Together these categories allow for the construction of a suitable 
bicategory $\SpnLens$, whose morphisms are spans in $\Lens$ but whose 
$2$-cells are defined by morphisms in $\Lens(B)$. 

\begin{defn}
An \emph{(asymmetric) lens} $(f, \varphi) \colon A \br B$ consists of a 
functor $f \colon A \rightarrow B$ together with a function, 
\[
	(a \in A, u \colon fa \rightarrow b) 
	\qquad \longmapsto \qquad
	\varphi(a, u) \colon a \rightarrow p(a, u)
\]
satisfying the axioms:
\begin{enumerate}[(1), itemsep=+1ex]
\item $f \varphi(a, u)= u$
\item $\varphi(a, 1_{fa}) = 1_{a}$
\item $\varphi(a, v \circ u) = \varphi(p(a, u), v) \circ \varphi(a, u)$
\end{enumerate}
Equivalently, an asymmetric lens consists of a functor 
$f \colon A \rightarrow B$ together with a cofunctor  
$\varphi \colon B \nrightarrow A$ such that $fa = \varphi a$ and 
$f \varphi(a, u) = u$. 
\end{defn}

The functor and cofunctor components of asymmetric lens are usually 
known as the \textsc{Get} and the \textsc{Put}, respectively. 
The three axioms above also correspond to the \textsc{PutGet}, 
\textsc{GetPut}, and \textsc{PutPut} laws, respectively. 

\begin{lemma}
\label{lem:lensrep}
Given a lens $(f, \varphi) \colon A \br B$ there is a 
commutative diagram of functors, 
\begin{equation}
\label{eqn:lensrep}
\begin{tikzcd}[column sep = small, row sep = small]
& \Lambda
\arrow[ld, "\varphi"']
\arrow[rd, "\phibar"]
& \\
A 
\arrow[rr, "f"']
& & B
\end{tikzcd}
\end{equation}
where $\varphi$ is an identity-on-objects functor and 
$\phibar$ is a discrete opfibration. 
\end{lemma}

Like cofunctors and Mealy morphisms, a lens will always be understood 
by its diagrammatic representation \eqref{eqn:lensrep} in $\Cat$. 
Let $\Lens$ be the category of small categories and lenses. 
Composition of lenses is given by composing the respective functor 
and cofunctor components, and the representation \eqref{eqn:lensrep} 
of the composite may be understood by the following diagram:
\begin{equation}
\begin{tikzcd}[row sep = small, column sep = small]
& &[-1em] \Lambda \times_{B} \Omega
\arrow[ld]
\arrow[rd]
\arrow[dd, phantom, "\lrcorner" rotate = -45, very near start]
&[-1em] & \\
& \Lambda 
\arrow[ld, "\varphi"']
\arrow[rd, "\phibar"]
& & \Omega
\arrow[ld, "\gamma"']
\arrow[rd, "\overline{\gamma}"]
& \\
A
\arrow[rr, "f"']
& & B
\arrow[rr, "g"']
& & C
\end{tikzcd}
\end{equation}

Given a pair of lenses $(f, \varphi) \colon A \br B$ and 
$(g, \gamma) \colon C \br B$ forming a cospan in $\Lens$, 
\begin{equation}
\label{eqn:lenscospan}
\begin{tikzcd}[row sep = small, column sep = small]
& \Lambda 
\arrow[ld, "\varphi"']
\arrow[rd, "\phibar"]
& & \Omega
\arrow[ld, "\overline{\gamma}"']
\arrow[rd, "\gamma"]
& \\
A
\arrow[rr, "f"']
& & B
& & C
\arrow[ll, "g"]
\end{tikzcd}
\end{equation}
there is a canonical cone, or ``fake pullback'', given by the span
in $\Lens$: 
\begin{equation}
\label{eqn:lenscone}
\begin{tikzcd}[row sep = small, column sep = tiny]
&[+1em] A \times_{B} \Omega
\arrow[ld, "\pi_{0}"']
\arrow[rd, "1 \times \gamma"]
& & \Lambda \times_{B} C
\arrow[ld, "\varphi \times 1"']
\arrow[rd, "\pi_{1}"]
&[+1em] \\
A
& & A \times_{B} C
\arrow[ll, "\pi_{0}"]
\arrow[rr, "\pi_{1}"']
& & C
\end{tikzcd}
\end{equation}

Note that this ``fake pullback'' diagram is sent to a genuine pullback 
via the forgetful functor $\Lens \rightarrow \Cat$. 
The category $\Lens$ also has the same terminal object as $\Cat$, and 
fake pullbacks over the terminal yields a semi-cartesian monoidal 
structure on $\Lens$. 

The reason \eqref{eqn:lenscone} fails, in general, to be a genuine
pullback in $\Lens$ is that the corresponding universal property is not 
satisfied. 
However, recall that pullbacks in $\Cat$ are the same as products in
a slice category $\Cat / B$ for some small category $B$. 
While the slice category $\Lens / B$ is not useful, there is a 
suitable category $\Lens(B)$ with cartesian products, 
together with a product-preserving functor 
$\Lens(B) \rightarrow \Cat / B$, that provides the ``fake pullbacks'' in 
$\Lens$ with a universal property. 

\begin{defn}
The category $\Lens(B)$ of \emph{lenses over a base category $B$} has 
objects given by lenses with codomain $B$, and morphisms 
$(f, \varphi) \rightarrow (g, \gamma)$ given by commutative diagrams of 
the form: 
\begin{equation}
\label{eqn:lensmorphism}
\begin{tikzcd}
\Lambda
\arrow[d, "\varphi"']
\arrow[rr, "\overline{h}"]
& & \Omega
\arrow[d, "\gamma"]
\\
A
\arrow[rr, "h"]
\arrow[rd, "f"']
& & C
\arrow[ld, "g"]
\\
& B 
\arrow[from = luu, crossing over, "\phibar" pos = 0.4] 
\arrow[from = ruu, crossing over, "\overline{\gamma}"' pos = 0.4] 
&
\end{tikzcd}
\end{equation}
\end{defn}
Note that only the functor $h \colon A \rightarrow C$ above need be 
specified; the functor $\overline{h} \colon \Lambda \rightarrow \Omega$ 
will always be uniquely induced from $h$, however may not (in general) 
make the back square in \eqref{eqn:lensmorphism} commute. 
The above definition of $\Lens(B)$ is motivated as a generalisation of
the category of $\SOpf(B)$ of split opfibrations and cleavage-preserving
functors, which is obtained as a full subcategory. 
An variant of $\Lens(B)$ has also been considered in \cite{JR13} as the 
category of algebras for a semi-monad on $\Cat / B$. 

\begin{proposition}
\label{prop:lensprod}
The category $\Lens(B)$ has products, for all small categories $B$. 
\end{proposition}
\begin{proof}
Consider a pair of lenses in $\Lens(B)$ as depicted in 
\eqref{eqn:lenscospan}.
Their product is given by the lens, 
\begin{equation}
\label{eqn:prodlens}
\begin{tikzcd}[column sep = tiny]
&
\Lambda \times_{B} \Omega
\arrow[rd, "\phibar\pi_{0} \ = \ \overline{\gamma}\pi_{1}"]
\arrow[ld, "\varphi \times \gamma"']
&[+1em] \\
A \times_{B} C
\arrow[rr, "f\pi_{0} \ = \ g \pi_{1}"']
& & B
\end{tikzcd}
\end{equation}
which is equal the composite of the appropriate lenses in 
\eqref{eqn:lenscospan} and \eqref{eqn:lenscone}. 
The product projections are given by the following diagrams: 
\begin{equation}
\begin{tikzcd}
\Lambda
\arrow[d, "\varphi"']
\arrow[from=rr, "\pi_{0}"']
& & \Lambda \times_{B} \Omega
\arrow[d, "\varphi \times \gamma"]
\\
A
\arrow[from=rr, "\pi_{0}"']
\arrow[rd, "f"']
& & A \times_{B} C
\arrow[ld, "f\pi_{0}"]
\\
& B 
\arrow[from = luu, crossing over, "\phibar" pos = 0.4] 
\arrow[from = ruu, crossing over, "\phibar\pi_{0}"' pos = 0.4] 
&
\end{tikzcd}
\qquad \qquad 
\begin{tikzcd}
\Lambda \times_{B} \Omega
\arrow[d, "\varphi \times \gamma"']
\arrow[rr, "\pi_{1}"]
& & \Omega
\arrow[d, "\gamma"]
\\
A \times_{B} C
\arrow[rr, "\pi_{1}"]
\arrow[rd, "g \pi_{1}"']
& & C
\arrow[ld, "g"]
\\
& B 
\arrow[from = luu, crossing over, "\overline{\gamma}\pi_{1}" pos = 0.4] 
\arrow[from = ruu, crossing over, "\overline{\gamma}"' pos = 0.4] 
&
\end{tikzcd}
\end{equation}
It is not difficult to show that the lens \eqref{eqn:prodlens} also 
satisfies the universal property of the product in the category 
$\Lens(B)$. 
\end{proof}

Proposition~\ref{prop:lensprod} shows that the fake pullbacks 
constructed in $\Lens$ actually have a universal property with respect
to the morphisms in $\Lens(B)$, for the appropriate small category $B$. 
Now consider the family of functors $\Lens(B) \rightarrow \Cat$ which 
assign each lens to its domain category, and each morphism 
\eqref{eqn:lensmorphism} to the corresponding functor between domains. 

\begin{defn}
Let $\SpnLens$ be the bicategory of \emph{spans of asymmetric lenses}, 
whose objects are small categories, and whose hom-categories 
$\SpnLens(A, B)$ are constructed by the pullback: 
\begin{equation}
\begin{tikzcd}[row sep = small, column sep = tiny]
& \SpnLens(A, B)
\arrow[ld]
\arrow[rd]
& \\
\Lens(A)
\arrow[rd]
& & \Lens(B)
\arrow[ld]
\\
& \Cat 
\arrow[phantom, from=uu, "\lrcorner" rotate = -45, very near start]
&
\end{tikzcd}
\end{equation}
An object in $\SpnLens(A, B)$ is a span of asymmetric lenses from 
$A$ to $B$, and a morphism is given by a functor $X \rightarrow X'$, 
together with induced functors $\Omega \rightarrow \Omega'$ and 
$\Lambda \rightarrow \Lambda'$, such that each face (including the two
outer squares) in the following diagram commute:
\begin{equation}
\label{eqn:lensspanmorp}
\begin{tikzcd}
\Omega 
\arrow[dd]
\arrow[rd]
\arrow[rr]
& & X
\arrow[ld]
\arrow[dd]
\arrow[rd]
& & \Lambda
\arrow[ll]
\arrow[dd]
\arrow[ld]
\\
& A
& & B & \\
\Omega'
\arrow[ru]
\arrow[rr]
& & X'
\arrow[lu]
\arrow[ru]
& & \Lambda'
\arrow[ll]
\arrow[lu]
\end{tikzcd}
\end{equation}
Horizontal composition is given by fake pullback of lenses, followed by 
lens composition of the projections with the appropriate legs of the 
span. 
Horizontal composition is associative up to natural isomorphism
with respect to the morphisms \eqref{eqn:lensspanmorp} above. 
\end{defn}

There is an identity-on-objects pseudofunctor 
$\Lens \rightarrow \SpnLens$ which takes a lens 
$A \br B$ to the right leg of a span of lenses
from $A$ to $B$, with left leg given by the identity lens.

The construction of the bicategory $\SpnLens$ is a generalisation of 
a \emph{category} previously defined in \cite{JR15, JR17art}. 
This category has objects given by small categories, and certain 
\emph{equivalence classes} of spans of asymmetric lenses as morphisms.
Removing the equivalence relation and considering the appropriate 
$2$-cells naturally gives rise to the bicategory $\SpnLens$ considered 
here.

%----------------------------------------------------------------------%
% Section 4: Symmetric Lenses        								   %
%----------------------------------------------------------------------%
\section{Symmetric lenses}
\label{sec:symmetriclenses}

The goal of this section is to introduce the bicategory 
$\SymLens$ of small categories and symmetric lenses. 

Consider the family of functors 
$U_{A, B} \colon \Meal(A, B) \rightarrow \Span(\Cat)(A, B)$ 
with the assignment on Mealy morphisms 
(Definition~\ref{defn:mealy}) stated for the span representation 
\eqref{eqn:Mealyrep} as follows:
\begin{equation}
\begin{tikzcd}[row sep = small, column sep = small]
& X 
\arrow[ld, "\overline{g}"']
\arrow[rd, "f"]
& \\
A & & B
\end{tikzcd}
\qquad \longmapsto \qquad
\begin{tikzcd}[row sep = small, column sep = small]
& X_{0} 
\arrow[ld, "g_{0}"']
\arrow[rd, "f_{0}"]
& \\
A & & B
\end{tikzcd}
\end{equation}
This functor is given by pre-composing the legs of the span with the 
canonical identity-on-objects functor from 
the discrete category $X_{0} \rightarrow X$. 
Furthermore, consider the family of functors
$\dagger_{A, B} \colon \Span(\Cat)(B, A) \rightarrow \Span(\Cat)(A, B)$
which send each span 
$(f, X, g) \colon B \nrightarrow A$ to its reverse span
$(g, X, f) \colon A \nrightarrow B$.

\begin{defn}
\label{defn:symmetriclens}
Let $\SymLens$ be the bicategory of \emph{symmetric lenses}, whose 
objects are small categories, and whose hom-categories $\SymLens(A,B)$ 
are constructed by the pullback: 
\begin{equation}
\begin{tikzcd}[row sep = small, column sep = tiny]
& \SymLens(A, B)
\arrow[ld]
\arrow[rd]
& \\
\Meal(A, B)
\arrow[rd, "U_{A, B}"']
& & \Meal(B, A)
\arrow[ld, "\dagger_{A, B}\circ U_{B, A}"]
\\
& \Span(\Cat)(A, B) 
\arrow[phantom, from=uu, "\lrcorner" rotate = -45, very near start]
&
\end{tikzcd}
\end{equation}
An object in $\SymLens(A, B)$ is a symmetric lens, and may be depicted
by a pair of spans: 
\begin{equation}
\label{eqn:symlens}
\begin{tikzcd}[row sep = small, column sep = small]
& X^{+} 
\arrow[ld, "\overline{g}"']
\arrow[rd, "f"]
& \\
A & & B
\\
& X^{-}
\arrow[lu, "g"]
\arrow[ru, "\overline{f}"']
& 
\end{tikzcd}
\end{equation}
The $2$-cells are given by the corresponding maps of Mealy morphisms, 
and horizontal composition is also inherited from $\Meal$. 
\end{defn}

\begin{notation}
In the diagram \eqref{eqn:symlens}, the upper span is a Mealy morphism 
$A \nrightarrow B$, while the lower span is a Mealy morphism 
$B \nrightarrow A$. 
As the notation suggests, both $X^{+}$ and $X^{-}$ are categories with 
the same discrete category of objects $X_{0}$. 
Moreover, the following diagrams commute:
\begin{equation}
\label{eqn:diagrams}
\begin{tikzcd}[row sep = small, column sep = small]
& X_{0}
\arrow[ld]
\arrow[rd, "{\langle g_{0}, f_{0} \rangle}"]
& \\
X^{+}
\arrow[rr, "{\langle \overline{g}, f \rangle}"']
& & A \times B
\end{tikzcd}
\qquad \qquad
\begin{tikzcd}[row sep = small, column sep = small]
& X_{0}
\arrow[ld]
\arrow[rd, "{\langle g_{0}, f_{0} \rangle}"]
& \\
X^{-}
\arrow[rr, "{\langle g, \overline{f} \rangle}"']
& & A \times B
\end{tikzcd}
\end{equation}
Taking these diagrams together with \eqref{eqn:symlens}, 
a symmetric lens may be completely described by the following 
commutative diagram of functors,
\begin{equation}
\begin{tikzcd}[row sep = small, column sep = small]
& X^{+} 
\arrow[ld, "\overline{g}"']
\arrow[rd, "f"]
& \\
A & 
X_{0}
\arrow[u]
\arrow[d]
& B
\\
& X^{-}
\arrow[lu, "g"]
\arrow[ru, "\overline{f}"']
& 
\end{tikzcd}
\end{equation}
where $\overline{g}$ and $\overline{f}$ are discrete opfibrations. 
However, for the remainder of the paper a symmetric lens will be depicted 
by a diagram of the form \eqref{eqn:symlens} for simplicity. 
\end{notation}

There is an identity-on-objects pseudofunctor 
$\Lens \rightarrow \SymLens$ with the following assignment on morphisms:
\begin{equation}
\label{eqn:lenstosym}
\begin{tikzcd}[row sep = small, column sep = small]
& \Lambda 
\arrow[ld, "\varphi"']
\arrow[rd, "\phibar"]
& \\
A
\arrow[rr, "f"']
& & B
\end{tikzcd}
\qquad \longmapsto \qquad
\begin{tikzcd}[row sep = small, column sep = small]
& A
\arrow[ld, "1"']
\arrow[rd, "f"]
& \\
A & & B
\\
& \Lambda
\arrow[lu, "\varphi"]
\arrow[ru, "\phibar"']
& 
\end{tikzcd}
\end{equation}
Note that the discrete opfibration $\phibar$ above would usually be 
denoted by $\overline{f}$ with the notational convention for symmetric 
lenses. 
From this pseudofunctor, symmetric lenses may be seen as a 
generalisation of asymmetric lenses. 
In $\SymLens$ morphisms are pairs of suitable Mealy morphisms, while in
$\Lens$ this must be a functor/cofunctor pair. 
However there is also a loss of information in \eqref{eqn:lenstosym}, 
as a symmetric lens no longer encodes the commutativity
condition of the corresponding asymmetric lens. 

The construction of the bicategory $\SymLens$ is a generalisation of 
a \emph{category} previously defined in \cite{JR15, JR17art}. 
This category has objects given by small categories, and certain 
\emph{equivalence classes} of symmetric lenses (called \emph{fb-lenses})
as morphisms.
Removing the equivalence relation and considering the appropriate 
$2$-cells yields the bicategory $\SpnLens$ considered here. 

%----------------------------------------------------------------------%
% Section 5: An adjoint triple										   %
%----------------------------------------------------------------------%
\section{An adjoint triple} 
\label{sec:adjointtriple}

This section presents the main theorem of the paper. 

\begin{theorem}
\label{thm:main}
Let $A$ and $B$ be small categories. 
Then there exists adjoint triple $L \dashv M \dashv R$ between the 
category of  symmetric lenses and the category of spans of asymmetric 
lenses,
\begin{equation}
\begin{tikzcd}[column sep = large]
\SymLens(A, B)
\arrow[r, shift right = 5, ""{name = D}, "R"', hook]
\arrow[from = r, ""{name = M}]
\arrow[r, shift left = 5, ""{name = U}, "L", hook]
& \SpnLens(A, B)
\arrow[phantom, from = M, to = D, "\bot"]
\arrow[phantom, from = M, to = U, "\bot"]
\end{tikzcd}
\end{equation}
such that $R$ is reflective and $L$ is coreflective (that is,  
$M L = M R = 1$).
\end{theorem}

The functor $M \colon \SpnLens(A, B) \rightarrow \SymLens(A, B)$ is 
defined on objects as follows:
\begin{equation}
\begin{tikzcd}
\Omega 
\arrow[rd, "\overline{\gamma}"']
\arrow[rr, "\gamma"]
&[-1em] & X
\arrow[ld, "g"]
\arrow[rd, "f"']
& &[-1em] \Lambda
\arrow[ll, "\varphi"']
\arrow[ld, "\phibar"]
\\
& A
& & B & \\
\end{tikzcd}
\qquad \longmapsto \qquad
\begin{tikzcd}[row sep = small, column sep = small]
& \Omega
\arrow[ld, "\overline{\gamma}"']
\arrow[rd, "f\gamma"]
& \\
A & & B
\\
& \Lambda
\arrow[lu, "g \varphi "]
\arrow[ru, "\phibar"']
& 
\end{tikzcd}
\end{equation}
Recall that $\gamma$ and $\varphi$ are identity-on-objects functors, 
so $\Omega$ and $\Lambda$ have the same objects, and the resulting 
symmetric lens is well-defined.  

To construct the right adjoint $R$, consider a symmetric lens given by
\eqref{eqn:symlens}.
Applying the bo-ff factorisation \eqref{eqn:boff} to the functor 
$\langle g_{0}, f_{0} \rangle \colon X_{0} \rightarrow A \times B$ 
yields a diagram: 
\begin{equation}
\begin{tikzcd}[row sep = small, column sep = small]
& \widetilde{X} 
\arrow[rd, "m"]
& \\
X_{0}
\arrow[ru, "e"]
\arrow[rr, "{\langle g_{0}, f_{0} \rangle}"']
& & A \times B
\end{tikzcd}
\end{equation}
This factorisation is chosen such that image $\widetilde{X}$ has the 
same objects as the discrete category $X_{0}$.  
Using the universal property \eqref{eqn:boffuniprop} of the bo-ff 
factorisation, together with the commutative diagrams 
\eqref{eqn:diagrams}, there exists unique, identity-on-objects functors: 
\begin{equation}
\begin{tikzcd}
X_{0} 
\arrow[d]
\arrow[r, "e"]
& \widetilde{X} 
\arrow[d, "m"]
\\
X^{+} 
\arrow[r, "{\langle \overline{g}, f \rangle}"']
\arrow[ru, dashed, "\sigma"]
& A \times B
\end{tikzcd}
\qquad \qquad 
\begin{tikzcd}
X_{0}
\arrow[d]
\arrow[r, "e"]
& \widetilde{X} 
\arrow[d, "m"]
\\
X^{-} 
\arrow[r, "{\langle g, \overline{f} \rangle}"']
\arrow[ru, dashed, "\tau"]
& A \times B
\end{tikzcd}
\end{equation}
The functor $R \colon \SymLens(A, B) \rightarrow \SpnLens(A, B)$ is 
defined on objects as follows:
\begin{equation}
\begin{tikzcd}[row sep = small, column sep = small]
& X^{+} 
\arrow[ld, "\overline{g}"']
\arrow[rd, "f"]
& \\
A & & B
\\
& X^{-}
\arrow[lu, "g"]
\arrow[ru, "\overline{f}"']
& 
\end{tikzcd}
\qquad \longmapsto \qquad
\begin{tikzcd}
X^{+} 
\arrow[rd, "\overline{g}"']
\arrow[rr, "\sigma"]
&[-1em] & \widetilde{X}
\arrow[ld, "\pi_{0} m"]
\arrow[rd, "\pi_{1} m"']
& &[-1em] X^{-}
\arrow[ll, "\tau"']
\arrow[ld, "\overline{f}"]
\\
& A
& & B & \\
\end{tikzcd}
\end{equation}

One may notice immediately that the composite 
$M R \colon \SymLens(A, B) \rightarrow \SymLens(A, B)$ is equal to 
the identity functor. 
The unit for the adjunction $M \dashv R$ is constructed using the 
universal property of the bo-ff factorisation, and is given as follows:
\begin{equation}
\begin{tikzcd}[row sep = small]
\Omega 
\arrow[dd, "1_{\Omega}"']
\arrow[rd, "\overline{\gamma}"']
\arrow[rr, "\gamma"]
& & X
\arrow[ld, "g"]
\arrow[dd, dashed]
\arrow[rd, "f"']
& & \Lambda
\arrow[ll, "\varphi"']
\arrow[dd, "1_{\Lambda}"]
\arrow[ld, "\phibar"]
\\
& A
& & B & \\
\Omega
\arrow[ru, "\overline{\gamma}"]
\arrow[rr, "\sigma"']
& & \widetilde{X}
\arrow[lu, "\pi_{0}m"']
\arrow[ru, "\pi_{1}m"]
& & \Lambda
\arrow[ll, "\tau"]
\arrow[lu, "\phibar"']
\end{tikzcd}
\end{equation}

To construct the left adjoint $L$, again consider a symmetric lens 
given by \eqref{eqn:symlens}. 
Since $X^{+}$ and $X^{-}$ have the same discrete category of objects 
$X_{0}$, there is pushout along the identity-on-objects functors 
given by: 
\begin{equation}
\begin{tikzcd}
X_{0}
\arrow[r]
\arrow[d]
& X^{-}
\arrow[d, "\iota_{1}"]
\\
X^{+}
\arrow[r, "\iota_{0}"']
& X^{+} \sqcup_{X_{0}} X^{-}
\arrow[from=lu, phantom, "\ulcorner", very near end]
\end{tikzcd}
\end{equation}
For brevity, let $\widehat{X} \coloneqq X^{+} \sqcup_{X_{0}} X^{-}$. 
Note that identity-on-objects functors are stable under pushout, so 
both $\iota_{0}$ and $\iota_{1}$ are also identity-on-objects functors.  
The functor $R \colon \SymLens(A, B) \rightarrow \SpnLens(A, B)$ is 
defined on objects as follows: 
\begin{equation}
\begin{tikzcd}[row sep = small, column sep = small]
& X^{+} 
\arrow[ld, "\overline{g}"']
\arrow[rd, "f"]
& \\
A & & B
\\
& X^{-}
\arrow[lu, "g"]
\arrow[ru, "\overline{f}"']
& 
\end{tikzcd}
\qquad \longmapsto \qquad
\begin{tikzcd}
X^{+} 
\arrow[rd, "\overline{g}"']
\arrow[rr, "\iota_{0}"]
&[-1em] & \widehat{X}
\arrow[ld, "{[\overline{g},\, g]}"]
\arrow[rd, "{[f,\, \overline{f}]}"']
& &[-1em] X^{-}
\arrow[ll, "\iota_{1}"']
\arrow[ld, "\overline{f}"]
\\
& A
& & B & \\
\end{tikzcd}
\end{equation}

One may notice immediately that the composite 
$M L \colon \SymLens(A, B) \rightarrow \SymLens(A, B)$ is equal to 
the identity functor. 
The counit for the adjunction $L \dashv M$ is constructed using the 
universal property of the pushout, and is given as follows:
\begin{equation}
\begin{tikzcd}
\Omega
\arrow[rd, "\overline{\gamma}"']
\arrow[rr, "\iota_{0}"]
\arrow[dd, "1_{\Omega}"']
&[-1em] & \widehat{X}
\arrow[dd, dashed, "{[\gamma,\, \varphi]}" description]
\arrow[ld, "{[\overline{\gamma},\, g\varphi]}"']
\arrow[rd, "{[f\gamma,\, \phibar]}"]
& &[-1em] \Lambda
\arrow[ll, "\iota_{1}"']
\arrow[ld, "\phibar"]
\arrow[dd, "1_{\Lambda}"]
\\[-0.5em]
& A
& & B & \\[-0.5em]
\Omega
\arrow[ru, "\overline{\gamma}"]
\arrow[rr, "\gamma"']
& & X
\arrow[lu, "g"']
\arrow[ru, "f"]
& & \Lambda
\arrow[ll, "\varphi"]
\arrow[lu, "\phibar"']
\end{tikzcd}
\end{equation}

\begin{corollary}
There exist identity-on-objects pseudofunctors between the bicategory
of symmetric lenses and the bicategory of spans of asymmetric lenses,
\begin{equation}
\begin{tikzcd}[column sep = large]
\SymLens
\arrow[r, shift right = 5, ""{name = D}, "R"']
\arrow[from = r, ""{name = M}, "M" description]
\arrow[r, shift left = 5, ""{name = U}, "L"]
& \SpnLens
\end{tikzcd}
\end{equation}
such that $L$ and $R$ are locally fully faithful and are locally adjoint
to $M$. 
\end{corollary}

%----------------------------------------------------------------------%
% Section 6: Concluding remarks										   %
%----------------------------------------------------------------------%
\section{Concluding remarks and future work} 
\label{sec:conclusion}

This paper has established a new category-theoretic foundation for 
symmetric delta lenses. 
In contrast to the algebraic approach of Johnson and Rosebrugh, this 
paper develops a natural diagrammatic approach to symmetric lenses and 
spans of asymmetric lenses, by using the properties of certain classes 
of functors. 
This framework yields significantly simpler definitions 
(for example, compare the characterisation of a symmetric lens via
Mealy morphisms in \eqref{eqn:symlens} to \cite[Definition~7]{JR17art}),
and allows for a clearer understanding of the composition of symmetric 
lenses, which is important for their application in fields such as 
database view-updating and model-driven engineering. 

While symmetric lenses and spans of asymmetric lenses were previously 
understood in \cite{JR17art} as morphisms in an \emph{isomorphic} pair 
of categories, the bicategories $\SymLens$ and $\SpnLens$ constructed in
this paper share a more interesting relationship.
The main theorem shows that $\SymLens(A, B)$ is both a reflective and 
coreflective subcategory of $\SpnLens(A, B)$, which suggests that 
symmetric lenses are less expressive than spans of asymmetric lenses
when modelling update propagation between systems.
The subcategory inclusions also provide a way of characterising
which spans of asymmetric lenses arise from symmetric lenses: 
either the functor components of the span are jointly fully faithful 
(via the right adjoint) or the identity-on-objects functors in the 
cofunctor components are pushout injections (via the left adjoint). 
A detailed study of the mathematical implications of the local 
adjunction between $\SymLens$ and $\SpnLens$ is left for future 
research. 

The notion of universal symmetric lenses, as considered in 
\cite{JR17pro, JR19}, will also be the focus of future work. 
In the paper \cite{Cla20}, explicit conditions for
universal asymmetric lenses were established, and it is hoped that these
results may be extended to the symmetric setting. 

Although this paper has established explicit technical advances towards
the understanding of symmetric delta lenses, it also suggests broader 
goals for the understanding of lenses in applied category theory. 
Analogous to the transition from functions to relations,  
this paper further develops the transition from asymmetric lenses to
the general setting of symmetric lenses, as pioneered by Johnson and 
Rosebrugh.
Realising this framework with other kinds of lenses in the 
literature has the potential to capture a wider range of applications
and deliver mathematically interesting results. 

%----------------------------------------------------------------------%
% Subsection: Acknowledgements                                         %
%----------------------------------------------------------------------%
\subsection*{Acknowledgements} 

The author is grateful to Michael Johnson and the anonymous reviewers 
for providing helpful feedback on this work. 
The author would also like to thank the organisers of the ACT2020 
conference.

%-----------------------------------------------------------------------
% References
%-----------------------------------------------------------------------
\bibliographystyle{eptcs}
\bibliography{bibliography-complete}

\begin{thebibliography}{10}
\providecommand{\bibitemdeclare}[2]{}
\providecommand{\surnamestart}{}
\providecommand{\surnameend}{}
\providecommand{\urlprefix}{Available at }
\providecommand{\url}[1]{\texttt{#1}}
\providecommand{\href}[2]{\texttt{#2}}
\providecommand{\urlalt}[2]{\href{#1}{#2}}
\providecommand{\doi}[1]{doi:\urlalt{http://dx.doi.org/#1}{#1}}
\providecommand{\bibinfo}[2]{#2}

\bibitemdeclare{phdthesis}{Agu97}
\bibitem{Agu97}
\bibinfo{author}{Marcelo \surnamestart Aguiar\surnameend}
  (\bibinfo{year}{1997}): \emph{\bibinfo{title}{Internal Categories and Quantum
  Groups}}.
\newblock Ph.D. thesis, \bibinfo{school}{Cornell University}.
\newblock \urlprefix\url{http://pi.math.cornell.edu/~maguiar/thesis2.pdf}.

\bibitemdeclare{inproceedings}{AU17}
\bibitem{AU17}
\bibinfo{author}{Danel \surnamestart Ahman\surnameend} \&
  \bibinfo{author}{Tarmo \surnamestart Uustalu\surnameend}
  (\bibinfo{year}{2017}): \emph{\bibinfo{title}{Taking Updates Seriously}}.
\newblock In: {\sl \bibinfo{booktitle}{Proceedings of the Sixth International
  Workshop on Bidirectional Transformations}}.
\newblock \urlprefix\url{http://ceur-ws.org/Vol-1827/paper11.pdf}.

\bibitemdeclare{inproceedings}{Cla19}
\bibitem{Cla19}
\bibinfo{author}{Bryce \surnamestart Clarke\surnameend} (\bibinfo{year}{2020}):
  \emph{\bibinfo{title}{Internal lenses as functors and cofunctors}}.
\newblock In: {\sl \bibinfo{booktitle}{Applied Category Theory 2019}},
  \bibinfo{series}{Electronic Proceedings in Theoretical Computer Science}.
\newblock \bibinfo{note}{(to appear)}.

\bibitemdeclare{misc}{Cla20}
\bibitem{Cla20}
\bibinfo{author}{Bryce \surnamestart Clarke\surnameend} (\bibinfo{year}{2020}):
  \emph{\bibinfo{title}{Internal split opfibrations and cofunctors}}.
\newblock \urlprefix\url{https://arxiv.org/abs/2004.00187}.

\bibitemdeclare{article}{DXC11}
\bibitem{DXC11}
\bibinfo{author}{Zinovy \surnamestart Diskin\surnameend},
  \bibinfo{author}{Yingfei \surnamestart Xiong\surnameend} \&
  \bibinfo{author}{Krzysztof \surnamestart Czarnecki\surnameend}
  (\bibinfo{year}{2011}): \emph{\bibinfo{title}{From State- to Delta-Based
  Bidirectional Model Transformations: The Asymmetric Case}}.
\newblock {\sl \bibinfo{journal}{Journal of Object Technology}},
  \doi{10.5381/jot.2011.10.1.a6}.

\bibitemdeclare{inproceedings}{DXCEHO11}
\bibitem{DXCEHO11}
\bibinfo{author}{Zinovy \surnamestart Diskin\surnameend},
  \bibinfo{author}{Yingfei \surnamestart Xiong\surnameend},
  \bibinfo{author}{Krzysztof \surnamestart Czarnecki\surnameend},
  \bibinfo{author}{Hartmut \surnamestart Ehrig\surnameend},
  \bibinfo{author}{Frank \surnamestart Hermann\surnameend} \&
  \bibinfo{author}{Fernando \surnamestart Orejas\surnameend}
  (\bibinfo{year}{2011}): \emph{\bibinfo{title}{From State- to Delta-Based
  Bidirectional Model Transformations: The Symmetric Case}}.
\newblock In: {\sl \bibinfo{booktitle}{Model Driven Engineering Languages and
  Systems}}, \doi{10.1007/978-3-642-24485-8_22}.

\bibitemdeclare{article}{FGMPS07}
\bibitem{FGMPS07}
\bibinfo{author}{J.~Nathan \surnamestart Foster\surnameend},
  \bibinfo{author}{Michael~B. \surnamestart Greenwald\surnameend},
  \bibinfo{author}{Jonathan~T. \surnamestart Moore\surnameend},
  \bibinfo{author}{Benjamin~C. \surnamestart Pierce\surnameend} \&
  \bibinfo{author}{Alan \surnamestart Schmitt\surnameend}
  (\bibinfo{year}{2007}): \emph{\bibinfo{title}{Combinators for Bidirectional
  Tree Transformations: A Linguistic Approach to the View-Update Problem}}.
\newblock {\sl \bibinfo{journal}{ACM Transactions on Programming Languages and
  Systems}}, \doi{10.1145/1232420.1232424}.

\bibitemdeclare{article}{HM93}
\bibitem{HM93}
\bibinfo{author}{Philip~J. \surnamestart Higgins\surnameend} \&
  \bibinfo{author}{Kirill C.~H. \surnamestart Mackenzie\surnameend}
  (\bibinfo{year}{1993}): \emph{\bibinfo{title}{Duality for base-changing
  morphisms of vector bundles, modules, Lie algebroids and Poisson
  structures}}.
\newblock {\sl \bibinfo{journal}{Mathematical Proceedings of the Cambridge
  Philosophical Society}}, \doi{10.1017/S0305004100071760}.

\bibitemdeclare{article}{HPW11}
\bibitem{HPW11}
\bibinfo{author}{Martin \surnamestart Hofmann\surnameend},
  \bibinfo{author}{Benjamin \surnamestart Pierce\surnameend} \&
  \bibinfo{author}{Daniel \surnamestart Wagner\surnameend}
  (\bibinfo{year}{2011}): \emph{\bibinfo{title}{Symmetric Lenses}}.
\newblock {\sl \bibinfo{journal}{ACM SIGPLAN Notices}},
  \doi{10.1145/1925844.1926428}.

\bibitemdeclare{inproceedings}{JR19}
\bibitem{JR19}
\bibinfo{author}{Michael \surnamestart Johnson\surnameend} \&
  \bibinfo{author}{Fran{\c c}ois \surnamestart Renaud\surnameend}
  (\bibinfo{year}{2019}): \emph{\bibinfo{title}{Symmetric c-Lenses and
  Symmetric d-Lenses are Not Coextensive}}.
\newblock In: {\sl \bibinfo{booktitle}{Proceedings of the 8th International
  Workshop on Bidirectional Transformations}}.
\newblock \urlprefix\url{http://ceur-ws.org/Vol-2355/paper7.pdf}.

\bibitemdeclare{inproceedings}{JR13}
\bibitem{JR13}
\bibinfo{author}{Michael \surnamestart Johnson\surnameend} \&
  \bibinfo{author}{Robert \surnamestart Rosebrugh\surnameend}
  (\bibinfo{year}{2013}): \emph{\bibinfo{title}{Delta Lenses and
  Opfibrations}}.
\newblock In: {\sl \bibinfo{booktitle}{Proceedings of the Second International
  Workshop on Bidirectional Transformations}},
  \doi{10.14279/tuj.eceasst.57.875}.

\bibitemdeclare{inproceedings}{JR14}
\bibitem{JR14}
\bibinfo{author}{Michael \surnamestart Johnson\surnameend} \&
  \bibinfo{author}{Robert \surnamestart Rosebrugh\surnameend}
  (\bibinfo{year}{2014}): \emph{\bibinfo{title}{Spans of lenses}}.
\newblock In: {\sl \bibinfo{booktitle}{Proceedings of the Workshops of the
  EDBT/ICDT 2014 Joint Conference}}.
\newblock \urlprefix\url{http://ceur-ws.org/Vol-1133/paper-18.pdf}.

\bibitemdeclare{inproceedings}{JR15}
\bibitem{JR15}
\bibinfo{author}{Michael \surnamestart Johnson\surnameend} \&
  \bibinfo{author}{Robert \surnamestart Rosebrugh\surnameend}
  (\bibinfo{year}{2015}): \emph{\bibinfo{title}{Spans of Delta Lenses}}.
\newblock In: {\sl \bibinfo{booktitle}{Proceedings of the Fourth International
  Workshop on Bidirectional Transformations}}.
\newblock \urlprefix\url{http://ceur-ws.org/Vol-1396/p1-johnson.pdf}.

\bibitemdeclare{inproceedings}{JR16}
\bibitem{JR16}
\bibinfo{author}{Michael \surnamestart Johnson\surnameend} \&
  \bibinfo{author}{Robert \surnamestart Rosebrugh\surnameend}
  (\bibinfo{year}{2016}): \emph{\bibinfo{title}{Unifying Set-Based, Delta-Based
  and Edit-Based Lenses}}.
\newblock In: {\sl \bibinfo{booktitle}{Proceedings of the Fifth International
  Workshop on Bidirectional Transformations}}.
\newblock \urlprefix\url{http://ceur-ws.org/Vol-1571/paper_13.pdf}.

\bibitemdeclare{article}{JR17art}
\bibitem{JR17art}
\bibinfo{author}{Michael \surnamestart Johnson\surnameend} \&
  \bibinfo{author}{Robert \surnamestart Rosebrugh\surnameend}
  (\bibinfo{year}{2017}): \emph{\bibinfo{title}{Symmetric delta lenses and
  spans of asymmetric delta lenses}}.
\newblock {\sl \bibinfo{journal}{Journal of Object Technology}},
  \doi{10.5381/jot.2017.16.1.a2}.

\bibitemdeclare{inproceedings}{JR17pro}
\bibitem{JR17pro}
\bibinfo{author}{Michael \surnamestart Johnson\surnameend} \&
  \bibinfo{author}{Robert \surnamestart Rosebrugh\surnameend}
  (\bibinfo{year}{2017}): \emph{\bibinfo{title}{Universal Updates for Symmetric
  Lenses}}.
\newblock In: {\sl \bibinfo{booktitle}{Proceedings of the Sixth International
  Workshop on Bidirectional Transformations}}.
\newblock \urlprefix\url{http://ceur-ws.org/Vol-1827/paper8.pdf}.

\bibitemdeclare{article}{LS02}
\bibitem{LS02}
\bibinfo{author}{Stephen \surnamestart Lack\surnameend} \&
  \bibinfo{author}{Ross \surnamestart Street\surnameend}
  (\bibinfo{year}{2002}): \emph{\bibinfo{title}{The formal theory of monads
  {II}}}.
\newblock {\sl \bibinfo{journal}{Journal of Pure and Applied Algebra}},
  \doi{10.1016/0022-4049(72)90019-9}.

\bibitemdeclare{article}{Par12}
\bibitem{Par12}
\bibinfo{author}{Robert \surnamestart Par\'{e}\surnameend}
  (\bibinfo{year}{2012}): \emph{\bibinfo{title}{Mealy Morphisms of Enriched
  Categories}}.
\newblock {\sl \bibinfo{journal}{Applied Categorical Structures}},
  \doi{10.1007/s10485-010-9238-8}.

\bibitemdeclare{article}{Str20}
\bibitem{Str20}
\bibinfo{author}{Ross \surnamestart Street\surnameend} (\bibinfo{year}{2020}):
  \emph{\bibinfo{title}{Polynomials as spans}}.
\newblock {\sl \bibinfo{journal}{Cahiers de Topologie et G\'{e}om\'{e}trie
  Diff\'{e}rentielle Cat\'{e}goriques}}
  \bibinfo{volume}{LXI}(\bibinfo{number}{2}), pp. \bibinfo{pages}{113--153}.

\end{thebibliography}

\end{document}